\documentclass[11pt]{amsart}

\usepackage[left=1in,top=1in,right=1in,bottom=1in,nohead]{geometry}
\usepackage{mathrsfs}
\usepackage{enumitem}
\usepackage{parskip}
\usepackage{indentfirst}
\usepackage{setspace}
\usepackage{color}
\newcommand{\red}[1]{\textcolor{red}{#1}}

\usepackage[OT2, T1]{fontenc}
\usepackage{url}
\usepackage{amsmath}
\usepackage{array}
\usepackage{graphicx}
\usepackage{amsfonts}
\usepackage{amssymb}
\usepackage{amstext}
\usepackage{amsthm}
\usepackage{hyperref}
\usepackage{colonequals}
\usepackage{enumitem}
\usepackage[alphabetic,lite]{amsrefs}
\usepackage[capitalise]{cleveref}
\usepackage[all,cmtip]{xy}
\usepackage{fullpage}
\usepackage{tikz-cd}
\usepackage{comment}

\numberwithin{equation}{subsection}

\newtheorem{theorem}{Theorem}[section]
\newtheorem{lemma}[theorem]{Lemma}
\newtheorem{proposition}[theorem]{Proposition}
\newtheorem{corollary}[theorem]{Corollary}

\theoremstyle{definition}
\newtheorem{defn}[theorem]{Definition}

\theoremstyle{remark}
\newtheorem{rem}[theorem]{Remark}

\newtheorem{example}[theorem]{Example}

\newcommand{\A}{\mathbb{A}}

\renewcommand{\C}{\mathbb{C}}

\newcommand{\E}{\mathbb{E}}

\newcommand{\G}{\mathbb{G}}
\newcommand{\Hh}{\mathbb{H}}

\newcommand{\Q}{\mathbb{Q}}

\newcommand{\R}{\mathbb{R}}

\newcommand{\Z}{\mathbb{Z}}
\newcommand{\ol}[1]{\overline{#1}}

\newcommand{\bbH}{\mathbf{H}}

\newcommand{\bbR}{\mathbf{R}}

\newcommand{\cF}{\mathcal{F}}

\newcommand{\cO}{\mathcal{O}}

\newcommand{\dR}{_{\mathrm{dR}}}

\newcommand{\ra}{\rightarrow}

\newcommand{\an}{{\mathrm{an}}}

\DeclareMathOperator{\GL}{GL}
\DeclareMathOperator{\SL}{SL}

\DeclareMathOperator{\Sp}{Sp}

\DeclareMathOperator{\Sym}{Sym}

\DeclareMathOperator{\End}{End}

\DeclareMathOperator{\Spec}{Spec}

\DeclareMathOperator{\Id}{Id}

\newcommand{\ds}{\displaystyle}

\def\trdeg{\operatorname{trdeg}}
\def\G{\mathbf{G}}
\def\AA{\mathbb{A}}

\def\Def{\mathrm{def}}
\def\Zar{\mathrm{Zar}}
\def\img{\operatorname{img}}

\def\GLGL{\mathbf{GL}}
\renewcommand{\bar}[1]{\overline{#1}}

\begin{document}
\title{Periods in Families and Derivatives of Period Maps}
\author{Benjamin Bakker, Jonathan Pila and Jacob Tsimerman}

\date{}

\maketitle

\begin{abstract}

    Given a smooth proper family $\phi:X\ra S$, we study the (quasi)-periods of the fibers of $\phi$ as (germs of) functions on $S$. We show that they field they generate has the same algebraic closure as that given by the flag variety co-ordinates parametrizing the corresponding Hodge filtration, together with their derivatives. Moreover, in the more general context of an arbitrary flat vector bundle, we determine the transcendence degree of the function field generated by the flat coordinates of algebraic sections. Our results are inspired by and generalize work of Bertrand--Zudilin.
    
\end{abstract}

%\begin{spacing}{0.01}
  \tableofcontents
%\end{spacing}

\section{Introduction}

\subsection{The cohomology groups of algebraic varieties}  

\subsubsection{Periods} Let $X$ be a smooth complex algebraic variety.  The singular cohomology groups $H^k(X^\an,\C)$ of the associated complex analytic variety $X^\an $ with its euclidean topology are topological invariants which come equipped with several algebraically defined structures.

By a theorem of Grothendieck, $H^k(X^\an,\C)$ can be computed using the de Rham cohomology of algebraic forms \cite{dmodules}.  This group is somewhat complicated in general, but if $X$ is affine then $H^k(X^\an,\C)$ is isomorphic to the group of closed algebraic $k$-forms modulo exact algebraic $k$-forms.  The isomorphism is by integration:  it associates to the class $[\alpha]$ of a closed algebraic $k$-form the integration map $\gamma\mapsto \int_\gamma \alpha$ on $k$-cycles.  The integrals $\int_\gamma\alpha$ are called periods.  When $X$ is defined over a subfield $L\subset \C$ (for example a number field) and $\alpha$ is taken to be $L$-rational, the periods are in general highly transcendental numbers, and conjecturally algebraic relations between them come from geometry (see for example \cite{AndreGroth} and the references therein).

\subsubsection{Hodge filtrations}
Algebraic de Rham cohomology also endows $H^k(X^\an,\C)$ with an algebraically defined filtration $F^\bullet$---the Hodge filtration---which detects geometry in a slightly different way, via the Hodge conjecture.  For $X$ a smooth projective variety, $F^p$ is given by the smooth de Rham cohomology classes which can be represented by a type $(r,s)$ form with $r\geq p$, and the Hodge conjecture asserts that rational cohomology classes $\alpha\in H^{2k}(X^\an,\Q)$ contained in $F^k$ should come from geometry.

In this paper we are concerned with the relationship between the periods and the Hodge filtration in algebraic families, specifically with regard to their transcendence properties.
\subsection{Periods in families} 
Suppose $f:X\to S$ is a smooth algebraic family over a smooth affine base $S$ which is real analytically locall trivial. Assume for simplicity that $X$ is affine.  Then we can take a fiberwise closed algebraic $k$-form $\alpha$, a point $s\in S$, a class $\gamma\in H_k(X_s^\an,\C)$, and a contractible neighborhood $s\in U
\subset S^\an$, and form an analytic function $U\to \C$ given by $u\mapsto \int_{\gamma(u)} \alpha|_{X_u}$ where $\gamma(u)\in H_k(X_u^\an,\C)$ is the cycle obtained from $\gamma$ by a local real analytic trivialization of $f$.  In the general case, these functions make sense for $\alpha$ an algebraic de Rham cohomology class (see for example \cite{huber}), and we define them to be the periods of $\alpha$ at $s$.  Once again the period functions are highly transcendental, but in this case algebraic relations can be proven to come from geometry in an appropriate sense, see \cite{AndreGroth}.

\subsection{Period maps}
See \cite{hodge1} for background on period maps. 
 The cohomology groups of fibers $H^k(X_t^\an,\C)$ can be canonically identified locally on $S^\an$ using a real analytic trivialization.  With respect to these local identifications, the Hodge filtration $F^\bullet_t$ on $H^k(X_t^\an,\C)$ varies in a holomorphic way.  By analytically continuing these identifications and using it to transport $F^\bullet_t$ to $H^k(X^\an_s,\C)$, we obtain a holomorphic map $\phi:\tilde X\to\check D$ where $\tilde X\to X^\an$ is a universal cover of $X^\an$ and $\check D$ is a flag variety parameterizing filtrations of the given type on $H^k(X_s^\an,\C)$.  If $f$ is quasiprojective, then in fact $\phi$ lands in a semialgebraic subset of an algebraic subvariety, namely the mixed period domain---see \cite{PSmix} for details.  Such a $\phi$ is called the period map associated to the natural variation of Hodge structures on $R^k(f^\an)_*\C_{X^\an}$.

If $X$ is once again affine, then around the basepoint $s$ we may take forms $\alpha_i$ which can be refined to a basis of each $F^p$; in general, the same can be done for algebraic de Rham cohomology.  Taking a lift of the neighborhood $U$ to $\tilde X$, the algebraic coordinates of $\phi$ restricted to $U$ are then rational functions in the ratios of certain minors of the matrix of periods $\int_{\gamma_j(u)}\alpha_i|_{X_u}$ with respect to a chosen basis of $\gamma_j$ of $H_k(X_s^\an,\C)$.  We call such functions the Hodge filtration coordinates at $s$.

\subsection{Comparison} Clearly the Hodge filtration coordinates are rational functions in the period, but a priori contain much less information.  For example, the field generated over the rational function field $K(S)$ by the periods is in fact closed under algebraic derivations - since the Lie derivative along an algebraic vector field is algebraic - whereas the field generated by the Hodge filtration coordinates are usually not differentially closed. On the other hand, the Hodge filtration coordinates are often more accessible. 

\subsection{Main results}

We now state a simplified version of our main theorem (Theorem \ref{KV in Kphi}), which clarifies the relationship between the fields generated by the periods and the Hodge filtration coordinates.
\begin{theorem}\label{thm:intro main}
Let $f:X\ra S$ be a smooth algebraic family for which $R^k(f^\an)_*\C_{X^\an}$ is a local system and let $\phi:\tilde X\to \check{D}$ be the period map associated to the natural variation of Hodge structures on $R^k(f^\an)_*\C_{X^\an}$.  Let $s\in S$ be a basepoint and consider the following two subfields of $K(\cO_{S^\an,s})$, the field of germs of meromorphic functions at $s$.
\begin{enumerate}
\item Let $K(R^k(f^\an)_*\C_{X^\an})_s$ be the subfield generated over $K(S)$ by the periods of degree $k$ algebraic de Rham cohomology classes at $s$.  
\item Let $K(\phi)_s$ to be the subfield generated over $K(S)$ by the Hodge filtration coordinates at $s$ and let $K_\partial (\phi)_s$ to be closure of $K(\phi)_s$ under algebraic derivations. 
\end{enumerate}
Then 

\begin{enumerate}[label=\alph*)]

    \item  $K_\partial(\phi)_s\subset K(R^k(f^\an)_*\C_{X^\an})_s$.
    \item Assume that the local system $R^k(f^\an)_*\C_{X^\an}$ does not admit a summand of the form $V_1\otimes V_2$, where $V_2$ is a unitary local system with infinite monodromy. Then the algebraic closures coincide: $\ol{K_\partial(\phi)_s}= \ol{K(R^k(f^\an)_*\C_{X^\an})_s}$.
\end{enumerate}
\end{theorem}
 \begin{rem}
    Condition (b) above is also necessary. Indeed, if $V$ is a complex variation of Hodge structures underlying a unitary local system, then the corresponding period map is constant, so the field $K_\partial(\phi)_s$ will be trivial. However, if the local system is non-trivial, then $V$ will give transcendental periods (see \S2). 
 \end{rem}

 \begin{rem}
     Condition (b), and therefore the equality of fields, holds for Shimura varieties corresponding to reductive groups without compact real factors.
 \end{rem}

The main theorem (see Theorem \ref{KV in Kphi}) generalizes Theorem \ref{thm:intro main} in several ways:
\begin{enumerate}
\item The field $K(R^k(f^\an)_*\C_{X^\an})$ may be replaced with $K(V)$ where $V$ is any complex local system on $S^\an$ with norm one eigenvalues at infinity.  Briefly, the field $K(V)$ is generated over $K(S)$ by the germs of the flat coordinates of algebraic sections of the canonical algebraic structure on the flat vector bundle $(\cO_{S^\an}\otimes_{\C_{S^\an}}V,\nabla)$ given by the Riemann--Hilbert correspondence.  This algebraic structure is uniquely associated to $V$, and in the case $V= R^k(f^\an)_*\C_{X^\an}$ recovers relative algebraic de Rham cohomology, which is why the notation only references the local system.  
\item The map $\phi$ may be replaced by any $\R_{\an,\exp}$-definable $\pi_1(X^\an,x)$-equivariant map $\phi:\tilde X\to Y^\an$, where $x\in X$ is a basepoint, $Y$ is an algebraic variety with a faithful action by a linear algebraic group $\G$ and the $\pi_1(X^\an,x)$-action on $Y$ is via a homomorphism $\rho :\pi_1(X^\an,x)\to\G(\C)$.
\end{enumerate}
In particular, the theorem applies more generally to any local system $V$ underlying a \emph{mixed} variation of \emph{complex} Hodge structures.

\subsection{Previous work}The paper is largely inspired by the work of Bertrand--Zudilin \cite{BZ}, who prove several results on the transcendence degree of fields generated by Siegel modular forms and their derivatives.  The moduli space of principally polarized abelian varieties can be identified with the quotient $A_g^\an=\Sp_{2g}(\Z)\backslash\mathbb{H}_g$, and Siegel modular forms can be interpreted as the algebraic coordinates on $A_g$ composed with the uniformization $\pi:\mathbb{H}_g\to A_g^\an$.  As such, they are the inverse functions to the Hodge filtration coordinates of the universal family of abelian varieties.  Theorem \ref{thm:intro main} then recovers the second part of \cite[Theorem 1]{BZ}.  Note that in the general context the inverse functions of the the Hodge filtration coordinates are not well-defined, since the period map will only be a local isomorphism in very special circumstances. 

\subsection{Outline of Paper}

In \S2, using the Riemann-Hilbert correspondence we define a Riemann-Hilbert field for an arbitrary local system which generalizes the notion of periods. In \S3 we relate definable meromorphic functions with finite-dimensional monodromy to the Riemann-Hilbert field of the corresponding local system, and we use this to prove the general form of our main theorem. We rely heavily on the Chow theorem of Peterzil--Starchenko \cite{definechow} and its generalization in \cite{bbt}. 
In \S4.1 we return to studying periods of algebraic families and prove our main result. We then work out in detail in \S4.2 the fields under question in the case of the universal elliptic curve. In \S4.3 we explain that even in the general case, periods of algebraic De-Rham cohomology classes can be thought of as integrals of meromorphic forms.

\section{Riemann--Hilbert fields of local systems}
Let $X$ be a smooth connected complex algebraic variety. Let $V$ be a complex local system on $X$. The associated locally free $\cO_{X^\an}$ with flat connection $(V_{\cO_{X^\an}}, \nabla)$ where $V_{\cO_{X^\an}}:=\cO_{X^\an}\otimes_{\C_X}V$ admits a canonical structure $(V_{\cO_X},\nabla)$ of a locally free $\cO_X$-module with flat connection with regular singularities by the Riemann--Hilbert correspondence \cite{deligneRH}.  We associate a field $K(V)$ generated by the components of algebraic sections of $V_{\cO_X}$ with respect to a flat basis as follows:

\begin{defn} Given a local system $V$ on $X$, for any point $x\in X$ we have a canonical evaluation $\langle\;,\;\rangle:V_{\cO,x}\otimes_\C V_x^\vee\to\cO_{X^\an,x}$ of $\cO_{X,x}$-modules, and we define $K(V)_x$ to be the subfield of the fraction field $K(\cO_{X^\an,x})$ generated by the image. 
\end{defn}

Concretely, $K(V)_x$ is the field generated by (germs of) the entries of the change of basis matrix between flat and algebraic frames of $V_{\cO_{X^\an}}$, and we refer to it as the Riemann--Hilbert field of $V$.  Note that a choice of (homotopy class of) path from $x$ to $y$ gives an isomorphism $K(V)_x\ra K(V)_y$ by analytic continuation.

We have the following basepoint-free definition of $K(V)_x$ given a universal cover:
\begin{enumerate}
    \item Let $V_{K(X)}$ be the $K(X)$-module of rational sections of $V_{\cO_X}$.  Let $\pi:\tilde X\to X^\an$ be a universal cover of $X^\an$ and $\tilde x$ a lift of $x$.  Then using the connection we obtain a canonical $\pi_1(X^\an,x)$-equivariant trivialization $\pi^*V^\vee\cong  \C_{\tilde X}\otimes_\C V_x^\vee$ and therefore a $\pi_1(X^\an,x)$-equivariant homomorphism $V_{K(X)}\otimes_\C V_x^\vee\to K(\tilde X)$ whose image is independent of $x$ (and $\tilde x$) and which is naturally isomorphic to $K(V)_x$ via pullback. We refer to this image as $K(V)\subset K(\tilde X)$; we suppress the dependence on $\tilde X$ as any isomorphism with any other universal cover respects the subfield $K(V)$.  Observe that $K(V)$ is stable under the action of $\pi_1(X^\an,x)$, and the action of $\pi_1(X^\an,x)$ on $K(V)$ factors through the monodromy representation of $V$.
   
    \item Let $\mathbb{E}:=\AA(V_{\cO_X}^\vee\otimes_\C V_x)$ be the geometric total space of the locally free sheaf $V_{\cO_X}^\vee\otimes_\C V_x$.  Concretely, $\mathbb{E}$ is the vector bundle of homomorphisms $V_y\to V_x$ as $y\in X$ in varies.  The sheaf $\mathbb{E}$ has a natural flat connection and there is a natural holomorphic map $T:\tilde X\to \mathbb{E}^\an$ sending $\tilde x$ to the identity which parameterizes the flat leaf through the identity.  In other words, thinking of $\tilde y\in \tilde X$ as a homotopy class of path from $y=\pi(\tilde y)\in X$ to $x$, $T_{\tilde y}:V_y\to V_{x}$ is the the parallel transport operator thought of as an element of $V_y^\vee\otimes V_x$. Note that the map $T$ depends on $\tilde{x}$, but the image of $T$ does not. 

    In coordinates the map $T$ is given as follows.  Let $t_i$ be a basis of $V_x$, $\tilde t_i$ the flat continuation of the $t_i$, $s_i$ a basis for $V_{K(X)}$.  Then at least over the open set of $X$ on which the sections $s_i$ are regular, with respect to the basis $s_i(y)$ on $V_y$ and $t_i$ on $V_x$ the matrix of $T_{\tilde{y}}$ has entries $\tilde t_j^\vee(\pi^*s_i)(\tilde y)$.  Thus, there is an natural isomorphism from $K((\img T)^\Zar)$ to $K(V)$ by pulling back along $T$, and using the identification from (1).
    
\end{enumerate}
In particular, we have the following:
\begin{lemma}
\[\trdeg_{K(X)}K(V)=\dim \,(\img T)^\Zar-\dim X.\]
\end{lemma}

We pause to record here some natural properties of $K(V)$.  Note that the tangent sheaf $T_X$ analytifies to $T_{X^\an}$ which pulls back to $T_{\tilde X}$ on the universal cover.  Thus, algebraic derivations of $X$ yield derivations on $X^\an$ and $\tilde X$.  
\begin{lemma}\label{basic RH}\hspace{1in}
\begin{enumerate}
    \item $K(V)\subset K(\tilde X)$ is closed under algebraic derivations of $X$.
    \item If $V'$ is a subquotient of $V$ then $K(V')\subset K(V)$.
    \item $K(V\otimes_\C V')\subset K(V)K(V')$. 
    \item $K(V^{\otimes n})\subset K(V)$ is algebraic.
    \item if $f:X'\to X$ is a dominant generically finite map of smooth varieties, then $K(V)\subset K(f^*V)$ is algebraic.
\end{enumerate}
\end{lemma}
\begin{proof}
\begin{enumerate}
    \item The connection on $V_{\cO_X}$ is algebraic.
    \item The Deligne canonical extension is functorial and $V_{\cO,x}'\otimes V'^\vee_x$ is a summand of $ V_{\cO,x}\otimes V_x^\vee$ (compatibly with $\langle\;,\;\rangle$), at least on a dense  Zariski open set.
    \item Obvious.
    \item For $\alpha_1,\dots, \alpha_n\in V_x$ and $s_1,\dots, s_n\in V_{\cO,x}$,  then $\ds\otimes_{i=1}^n\alpha_i(\ds\otimes_{j=1}^n s_j) = \ds\prod_{k=1}^n \alpha_k(s_k)\in K(V)_x$.
    \item We may pull-back flat and algebraic frames from $X$, and the entries of their change of basis matrix generates $K(f^*V)$ over $K(X')$.
\end{enumerate}
\end{proof}

We now relate the transcendence degree of $K(V)/K(X)$ to the Zariski closure $\G$ of the image $\Gamma$ of the monodromy representation $\pi_1(X^\an,x)\to\End_\C(V_x)$. This is a known consequence of differential Galois theory (see \cite[V.Cor.1]{andredgg}) but we include a proof more in the spirit of our geometric approach.

\begin{proposition}
We have 
\[\trdeg_{K(X)}K(V) = \dim\G\]
\end{proposition}
 \begin{proof}
 Again, let $t_i$ be a basis of $V_x$ and $\tilde t_i$ its flat continuation.  We identify $\E^\an$ as the total space of the local system $V^\vee\otimes_\C V_x$, and the universal cover of $\E^\an$ as $\tilde X\times \End(V_x)$ via the map $q:\tilde X\times \End(V_x)\to \E^\an$ sending $(\tilde y,f)\mapsto (y,f\circ T_{\tilde y})$.  Via this identification, we get an induced monodromy action on $\tilde X \times \End(V_x)$:  explicitly $\gamma\in\pi_1(X^\an,x)$ maps $(\tilde y,f)\mapsto (\gamma \tilde y,f\circ T_\gamma^{-1})$.  Thus if we let $s:X\ra X\times \End(V_x)$ denote the identity section, then $T=q\circ s$.
 
Let $Z$ be the Zariski closure of $\img T$. Note that $q^{-1}(\img T)$ is simply $\tilde X\times \Gamma$.  As $q$ is an algebraic isomorphism on fibers and $\Gamma$ is Zariski dense in $\G$ it follows that $q^{-1}(Z)\supset \tilde X\times G(\C)$.  Let $Q$ be the image of $\tilde X\times G(\C)$ in $\E^{\an}$. We will show that $Q$ is algebraic, and therefore that $Z=Q$ is a $G$-bundle over $X$, `completing the proof. 

To see this, first note using Chevalley's theorem \cite[Theorem 4.19]{milnelag} that there is a line $L_x$ inside a representation $W_x$ of $\GL(V_x)$ such that $\G$ is the stabilizer of $L_x$. Note that $W_x$ occurs as a direct summand of a direct sum of tensor powers of $V_x$ and its dual. It follows that $W_x$ corresponds to the fiber at $x$ of a local system $W$ occcuring as a direct summand of a direct sum of tensor powers of $V$ and its dual, and $L$ corresponds to a rank one sub local system of $W$.

It remains to observe that $Q$ consists of those elements $f:V_y\ra V_x$ of $E^{\an}$ such that $f$ induces an isomorphism from $L_y$ to $L_x$. Since the Riemann-Hilbert correspondence is functorial, this latter condition is algebraic, hence $Q$ is algebraic as desired.

 \end{proof}

\section{Local systems of holomorphic functions}
\subsection{Universal covers and definability}
Throughout this section, by definable we always mean definable in the o-minimal structure $\R_{\an,\exp}$; see \cite{Dries} for details.  Let $X$ be a definable (complex) analytic space (for example, the definabilization of a complex algebraic variety).  We can think of $X$ as locally the zero-set of definable holomorphic functions, with definable holomorphic gluing functions (see \cite{bbt} for full details).  Let $\pi:\tilde X\to X^\an$ be the universal cover of the associated analytic space.  We say a meromorphic function $g\in K(\tilde X)$ is definable if for any open definable $f:U\to X$ and any analytic lift $\tilde f:U^{an} \to \tilde X$ the pullback $\tilde f^*g$ is definable.  We denote by $K_\Def(\tilde X)\subset K(\tilde X)$ the field of definable meromorphic functions on $\tilde X$.  Likewise, for any definable analytic space $Y$ we say an analytic map $\phi:\tilde X\to Y^\an$ is definable if for any open definable $f:U\to X$ and any analytic lift $\tilde f:U^\an\to \tilde X$ the composition $\phi\circ\tilde f$ is definable.\footnote{ Note that we are not endowing $\tilde X$ with the structure of a definable analytic space, since it will not in general be covered by finitely many lifts of open definable sets in $X$; this will only happen if $\pi_1(X)$ is finite.}

% \begin{lemma}Let $f_i:U_i\to X$ be definable analytic morphisms with analytic lifts $\tilde f_i:U_i^\an\to\tilde X$ for $i=1,2$.  Then the natural map $U_1^\an\times_{\tilde X}U_2^\an\to U_1^\an\times U_2^\an$ is the analytification of a definable analytic morphism.
% \end{lemma}

% \begin{proof}
% One easily checks that fiber products exist in the category of definable analytic spaces.  The diagonal $\Delta$ is a connected component of $\tilde X\times_X\tilde X\subset \tilde X\times \tilde X$.  Therefore, the map $U_1^\an\times_{\tilde X}U_2^\an\to U_1^\an\times U_1^\an$ is the composition of $U_1^\an \times_{\tilde X}U_2^\an\to U_1^\an\times_X U_2^\an$, which is the embedding of a union of connected components, and $U_1^\an\times_XU_2^\an\to U_1\times U_2$, both of which are definable.
% \end{proof}

By a definable fundamental set for $X$ we mean a definable analytic space $\cF$, a surjective \'etale definable analytic morphism $f:\cF\to X$ and an analytic lift $\tilde f:\cF^\an\to \tilde X$.  

\begin{lemma} Let $f:\cF\ra X$ be a definable fundamental set.

\begin{enumerate}
    \item A function $g\in K(\tilde X)$ is in $K_\Def(\tilde X)$ if and only if the pullback $(\gamma\circ \tilde f)^*g$  is definable for each $\gamma\in \pi_1(X^\an,x)$.
    \item A map $\phi:\tilde X\to Y^\an$ is definable if and only if the map $\phi\circ\gamma\circ\tilde f$ is definable for each $\gamma\in \pi_1(X^\an,x)$.
\end{enumerate}
\end{lemma}

\begin{proof}

\begin{enumerate}
    \item For the necessity, it is sufficient to take $\gamma=\Id$. let $U_1,\dots,U_n$ be finitely many connected definable open sets in $X$ whose analytifications lift via $\iota_i$ to $\tilde{X}$. Let $V_j$ be the connected components of $f^{-1}(U_i)$.  Then for each $V_j$ mapping to $U_i$, one may pick an element $\gamma'\in \pi_1(X^{\an},x)$ such that $\tilde{f}\mid V_j = \gamma'\circ\iota_i\circ f\mid V_j$. Now since $g$ is definable it follows by definition that $(\gamma\circ\iota_i)^*g$ is definable, and hence so is $\tilde{f}^* g\mid V_j$. The claim now follows since the $V_j$ cover $\cF$.
    
    For the sufficiency, let $U\subset X$ be definable open set and $\iota:U^{an}\ra \tilde{X}$ be a lift. Let $V=f^{-1}U$ and let $V_1,\dots,V_m$ be the connected components of $V$. Then for each $i$ there is an element $\gamma_i\in\pi_1(X^{\an},x)$ such that $\gamma_i\circ f\mid V_i = \iota\circ f\mid V_i$. Now by assumption $(\gamma_i\circ f)^*g $ is definable, and therefore it follows that so is $(\iota\circ f\mid V_i)^*g$, and hence so is $(\iota\circ f\mid V)^* g$. Now, since $f\mid V$ is surjective onto $U$, it follows by definable choice that $\iota^*g$ is also definable, as desired.

    \item This is the exact same proof as above, replacing $g$ by $\phi$ at each step.
\end{enumerate}
    
\end{proof}

\begin{example}
By definable triangulation \cite{Dries}, any definable analytic space $X$ admits a cover by simply connected definable open sets $U_i\subset X$, and for arbitrary choices of lifts $\tilde U_i\subset\tilde X$ the union $\cF=\bigcup_i U_i$ (together with the lift) is a definable fundamental set for $X$.
\end{example}
\begin{example}Let $X$ be a smooth algebriac variety.  We cover a log smooth compactification $\bar X$ with polydisks $\Delta^n\cong U_i\subset \bar X$ for which $\Delta^{k_i}\times(\Delta^*)^{\ell_i}\cong U_i\cap X$.  Let $\Sigma\subset\Hh$ be a bounded vertical strip fundamental set for the exponential $e^{2\pi i z}:\Hh\to\C^*$ endowed with the obvious definable structure.  For each $i$ choose a lift $f_i:\Delta^{k_i}\times\Sigma^{\ell_i}\to \tilde X$ of the composition $\Delta^{k_i}\times \Sigma^{\ell_i}\to\Delta^{k_i}\times(\Delta^*)^{\ell_i}\cong U_i\subset X$. Then after shrinking each polydisk, the union of the images of the $f_i$ form such a definable fundamental set.
\end{example}

\subsection{Local systems of definable meromorphic functions}

Assume now that $X$ is a smooth complex algebraic variety with its canonical definable structure, and let $x\in X$ be a choice of basepoint.  We denote by $X^{\Def}$ the definabilization of $X$, which is a definable analytic space (See \cite{bbt}). For a complex local system $V$ on $X$, we say $V$ has norm one eigenvalues at infinity if for some (hence any) log smooth compactification $\bar X$ of $X$, the local monodromy of $V$ at the boundary has eigenvalues with (complex) norm $1$. In general, the total space of $V$ has two definable structures: the flat one and the algebraic one (given by Riemann-Hilbert). By \cite{BMull}, if $V$ has norm one eigenvalues at infinity then these two definable structures on $V$ are equivalent.  Observe that:
\begin{enumerate}
    \item $K_\Def(\tilde X)$ is closed under algebraic derivations of $X$.
    \item Up to passing to a Zariski open set, a homomorphism of definable sheaves $\phi:V\to \cO_{X^\Def}$ for a local system $V$ is equivalent to a homomorphism of $\pi_1(X^\an,x)$-modules $V_x\to K_\Def(\tilde X)$.
    \item If $V$ has norm one eigenvalues at infinity, then $K(V)\subset K_\Def(\tilde X)$.
\end{enumerate}

\begin{lemma}\label{lem: monofactor}Suppose $V$ has norm one eigenvalues at infinity.  Then any homomorphism of $\pi_1(X^\an,x)$-modules $\mu:V_x^\vee\to K_\Def(\tilde X)$ factors through $K(V)$. 
\end{lemma}
\begin{proof}After passing to a Zariski open, let $\nu :V^\vee\to \cO_{X^\Def}$ be the resulting homomorphism of definable sheaves.  Note that the canonical algebraic structure on $(V^\vee)_{\cO_{X^\an}}$ is canonically $(V_{\cO_X})^\vee$.
We have a homomorphism of definable coherent sheaves $\mathrm{id}\otimes\nu: \cO_{X^\Def}\otimes_{\C_{X^\Def}}V^\vee \to\cO_{X^\Def}$ which by definable GAGA \cite{bbt} comes from an algebraic homomorphism $F:V_{\cO_X}^\vee\to \cO_X$ which in turn is equal to evaluation on a rational section $s$ of $V_{\cO_X}$.  Thus, $\mu$ is identified with $v\ra \langle s,v\;\rangle$, which completes the proof.
\end{proof}
The main idea in the proof of the lemma also gives us a criterion for a subfield $L\subset K_\Def(\tilde X)$ to contain $K(V)$.  In general for a $G$-module $U$ we denote by $U_0$ the same vector space with trivial $G$-action.  
\begin{lemma}\label{condition}  Suppose $K(X)\subset L\subset K_\Def(\tilde X)$ is a $\pi_1(X^\an,x)$-stable subfield and that there is a homomorphism of $\pi_1(X^\an,x)$-modules
\[V_x^\vee\to L\otimes_\C (V_x^\vee)_0\]
which evaluates to an isomorphism upon specialization of $K$ to some point of $X$.  Then $K(V)\subset L$.

\end{lemma}
\begin{proof}
As in the previous lemma, we obtain a map of definable sheaves $V^\vee\to\cO_{X^\Def}\otimes_\C (V_x^\vee)_0$ and therefore a homomorphismm $F:\cO_{X^\Def}\otimes_{\C_{X^\Def}}V^\vee \to\cO_{X^\Def}\otimes_\C (V_x^\vee)_0$ which is algebraic by definable GAGA.  It has full rank at some point of $X$ and is therefore rationally an isomorphism.  On the one hand, choosing a rational basis $s_i$ and a flat basis $1\otimes t_j$ of the source $\cO_{X^\Def}\otimes_{\C_{X^\Def}}V^\vee$, the generators of $K(V)$ are the matrix elements of the change of basis matrix between the $s_i$ and the $1\otimes t_j$.  One the other hand, the images $F(1\otimes t_j)$ are in $L\otimes_\C (V^\vee_x)_0$ and the images $F(s_i)$ are in $K(X)\otimes_\C(V^\vee_x)$, so the change of basis matrix between $F(s_i)$ and the $F(1\otimes t_j)$ has entries in $K$.  Thus, $K(V)\subset L$.
\end{proof}

\subsection{Generalized period maps}

Suppose we have a homomorphism $\rho:\pi_1(X^\an,x)\to \G(\C)$.  Let $\Gamma\subset\G(\C)$ be the image of $\pi_1(X^\an,x)$.  Suppose further that $Y$ is an algebraic variety with an action of $\G$ with finite kernel, meaning no positive-dimensional subgroup of $\G(\C)$ acts as the identity on all of $Y$.  Let $\phi:\tilde X\to Y^\an$ be a definable $\pi_1(X^\an,x)$-equivariant map.  

\begin{defn}
We define $K_\partial(\phi)$ to be the subfield of $K_\Def(\tilde X)$ generated under algebraic derivations of $X$ by pullbacks of rational functions of $Y$ which are generically regular along $\img \phi$ (or equivalently of $K(\img\phi)^\Zar$).
\end{defn}  Note that $K_\partial(\phi)$ contains $K(X)$.

\begin{rem}
By definable Chow, a definable $\pi_1(X^\an,x)$-equivariant map $\tilde X\to Y^\an$ up to passing to a dense Zariski open set is equivalent to an element of $Y(K_\Def(X))^{\pi_1(X,x)}$, that is, a $\pi_1(X^\an,x)$-invariant $K_\Def(X)$-rational point of $Y$.  
\end{rem}

\begin{lemma}\label{et lift}Suppose $f:Y'\to Y$ is a $\G$-equivariant finite map of varieties with $\G$-action and that we have a commutative diagram
\[
\begin{tikzcd}
&Y'^\an\ar{d}{f^\an}\\
\tilde X\ar{ur}{\phi'}\ar{r}{\phi}&Y^\an
\end{tikzcd}
\]
where $\phi$ and $\phi'$ are definable $\pi_1(X^\an,x)$-equivariant.  Then $K_\partial(\phi)\subset K_\partial(\phi')$ is algebraic.
\end{lemma}

\begin{proof}We may assume the image of $\phi$ (hence of $\phi'$) is Zariski dense.  Any function $g'$ on $Y'$ satisfies a polynomial $P\in K(Y)[t]$, and $(\phi')^*g'$ satisfies the polynomial $\phi^*P\in K_\partial(\phi)[t]$.  Let $K$ be the field generated by the $\phi'^*g'$ over $K_\partial(\phi)$.  It remains to show that $K$ is closed under algebraic derivations.

Any $h\in K$ satisfies a minimal polynomial $P(t)$ over $K_\partial(\phi)$.  For any algebraic derivation $\theta$ of $X$, we then have
\[0=(\theta P) (h)+P'(h)\theta h\]
where $\theta P$ is the polynomial with differentiated coefficients, and $P'(t)$ is the formal derivative with respect to $t$ (treating the coefficients as constant).  Since $0\neq P'(h)$, we have $\theta h\in K$.
\end{proof}

\begin{theorem}\label{KV in Kphi}Let $Y$ be an algebraic variety with an action of an algebraic group $\G$ with finite kernel in the above sense.  Let $U$ be a representation of $\G$ with finite kernel, $\rho:\pi_1(X^\an,x)\to \G(\C)$ a homomorphism, $V$ the resulting local system on $X$, and $\phi:\tilde X\to Y^\an$ a definable $\pi_1(X^\an,x)$-equivariant map with Zariski dense image.  Then if $V$ has norm one eigenvalues at infinity,
\[\bar{K(V)}= \bar{K_\partial(\phi)}.\]
\end{theorem}

Before the proof we recall jet spaces.  Let $A$ be an artinian $\C$-algebra.  Recall that for a $\C$-scheme $Y$ the jet space $J_AY$ parametrizes $\C$-morphisms $\Spec(A)\to Y$.  Precisely, $\C$-morphisms $S\to J_AY$ are $\C$-morphisms $\Spec(A)\times S\to Y$.  For any homomorphism $A'\to A$ of artinian $\C$-algebras there is a natural morphism $J_{A'}Y\to J_{A}Y$, and in particular there is a morphism $\pi_A:J_AY\to Y$ by taking the canonical quotient $A\to \C$.   Note that if $Y$ is smooth, then for any $A$ the fibers of $\pi_A$  are irreducible and isomorphic, hence $J_AY$ is irreducible.  Moreover, for any small extension $A'\to A$ with ideal $I$, $J_{A'}Y$ has a canonical action over $J_AY$ by $\pi_A^*T_Y\otimes_\C I$ coming from the isomorphism $A'\times_A A'\to k\langle I\rangle\times_\C A':(x,y)\mapsto (x-y,y)$ where $k\langle I\rangle$ is the ring of dual numbers with tangent space $I^\vee$.  For any section $\theta$ of $\pi^*_AT_Y\otimes_\C I$ we denote its action by $t_\theta:J_{A'}Y\to J_{A'}Y$, so for a point $\xi:\Spec(A)\to Y$ of $J_AY$ the pull-back map along $t_\theta(\xi)$ on functions is given by $t_\theta(\xi)^*=\xi^*+\theta:\cO_X\to A$.  This makes $J_{A'}Y\to J_AY$ into a torsor in the Zariski topology, since nilpotent thickenings of affine schemes lift through smooth morphisms and therefore $J_{A'}Y\to J_AY$ has a section Zariski locally.  In particular, $\pi_A:J_AY\to Y$ is affine.  

For any algebraic (resp. analytic) function $g$ on $Y$ and $\alpha\in A^\vee$ (the dual as a complex vector space), we obtain an algebraic (resp. analytic) function $d^\alpha g$ on $J_AY$ which evaluates on a map $\xi:\Spec(A)\times S\to Y$ as $(\alpha\otimes\mathrm{id})(\xi^*g)\in\cO_S(S)$.  Clearly for any algebraic or analytic map $f:Y'\to Y$ and the induced map $J_Af:J_AY'\to J_AY$ we have $(J_Af)^*d^\alpha g=d^\alpha (f^*g)$.

The following lemma says in particular that for smooth $Y$, up to passing to a Zariski cover, any map of jet spaces admits a section through any point.
\begin{lemma}
Suppose $Y$ is affine and $q:Y\to \A^n$ is \'etale.  Let $z_1,\ldots,z_n$ be the pullbacks of the coordinates on $\A^n$.  Then:
\begin{enumerate}
    \item For any surjective homomorphism $A'\to A$ of artinian rings, the resulting map $p:J_{A'}Y\to J_AY$ has a section through any point of $J_{A'}Y$.
    \item If moreover $A'\to A$ is small with ideal $I$ with $\dim_\C I=1$ we have $p_*\cO_{J_{A'}Y}\cong \cO_{J_AY}[d^{\alpha}z_1,\ldots,d^\alpha z_n]$ for any $\alpha\in A^\vee$ which is nonzero on $I$.
\end{enumerate}
\end{lemma}
\begin{proof}
For the first part, any point of $J_AY$ can be pushed down to $\AA^n$, translated to a section of $J_A\AA^n$, and lifted (uniquely) to $J_AY$, as maps from nilpotent thickenings of affine schemes lift uniquely through \'etale maps.  For the second part, its enough to observe that the natural derivations $\partial_i:=\partial/\partial z_i$ give a trivialization of $\pi_A^*T_Y\otimes_\C I$ and $t_\theta^*d^\alpha g=d^\alpha g+\pi_A^*\alpha(\theta g) $, so the $d^\alpha z_i$ necessarily give the linear affine coordinates of a trivialization $J_{A'}Y\cong J_A Y\times\A^n$ of the torsor structure. 
\end{proof}

\begin{lemma}
Suppose we have a section $\sigma:X\to J_AX$ of $\pi_A:J_AX\to X$.  Then for any function $g$ on $X^\an$, $(\sigma^\an)^*d^\alpha g$ is in the subsheaf of $\cO_{X^\an}$ generated over $\cO_X$ by $(\sigma^\an)^*g$ and its algebraic derivatives.
\end{lemma}
\begin{proof}The claim can be checked Zariski locally on $X$, so we may assume $X$ is affine with an \'etale map $q:X\to\A^n$.  Let $z_i$ be the pullbacks of the coordinate functions and $\partial_i$ the associated derivations as in the previous lemma, so that the functions $z_i-z_i(y)$ generate $m_y$ at every point $y\in X$.  We then have a universal Taylor series in the following sense.  Denoting the $i$th projection $\pi_i:X\times X\to X$ we have
\begin{equation}\label{taylor}\pi_1^*g-\pi_2^*g=\sum_{J}\frac{\pi_2^*(\partial_Jg)}{J!}(\pi_1^*z-\pi_2^*z)^J\;\;\; \mbox{in}\;\;\;\varprojlim_k\cO_{X\times X}/ I_{\Delta}^k.\end{equation}
in the obvious notation, where $\Delta\subset X\times X$ is the diagonal.  The same formula holds for analytic $g$ as well.

Suppose $\sigma$ is given by a map $\xi:\Spec(A)\times X\to X$ and let $p_2:\Spec(A)\times X\to X$ be the second projection.  Then $\xi\times p_2:\Spec(A)\times X\to X\times X$ factors through $\Spec(\cO_{X\times X}/I_\Delta^k)$ for some $k$, and the claim follows from \eqref{taylor} since we then have
\[(\xi^\an)^*g-1\otimes g=\sum_{J}\frac{1\otimes\partial_Jg}{J!}(\xi^*z-1\otimes z)^J\;\;\; \mbox{in}\;\;\; A\otimes_\C\cO_{X^\an}.\]
\end{proof}
\begin{corollary}\label{jet lift}
Suppose we have a section $\sigma:X\to J_AX$ of $\pi_A:J_AX\to X$ and a commutative diagram
\[\begin{tikzcd}
J_A\tilde X\ar{d}{\tilde \pi_A}\ar{r}{J_A\phi}&J_AY^\an\ar{d}{ \pi_A^\an}\\
\tilde X\ar{r}{\phi}\ar[bend left]{u}{\tilde\sigma}&Y^\an
\end{tikzcd}
\]
where $\phi$ (hence $J_A\phi$) is definable $\pi_1(X^\an,x)$-equivariant.  Then $K_\partial(J_A\phi \circ\tilde\sigma)\subset K_\partial(\phi)$.
\end{corollary}

\begin{comment}
\begin{lemma}
In the above setup,
\begin{enumerate}
\item For $g:X'\to X$ surjective, $\tilde g\circ f:\tilde X'\to Y$ is definable and \[K(f)\subset K(\tilde g\circ f).\]
    \item For any artinian $A$, $J_A(f):J_A(\tilde X)\to J_A(Y)$ is definable and
    \[K(J_A(f))=K(f).\]
    \item Suppose we have $g:Z\to J_A(X)$ whose composition $Z\to X$ is generically finite.  Then $J_A(f)\circ \tilde g:\tilde Z\to J_A(Y)$ is definable and 
    \[\overline{K(J_A(f)\circ \tilde g)}=\overline{K(f)}.\]
\end{enumerate}
\end{lemma}
\end{comment}

\begin{proof}[Proof of \cref{KV in Kphi}]
Observe that we are free to replace $X$ with a dense Zariski open subset. Note that if $Z\subset Y$ is a closed $\G$-invariant subvariety then $\phi^{-1}(Z)$ is a definable closed $\pi_1(X)$-invariant subset of $\tilde{X}$, and therefore descends to a definable analytic subvariety of $X$, which is algebraic by   definable Chow \cite{definechow}. Thus, we may also freely replace $Y$ with a $\G$-invariant dense Zariski open subset (at the cost of replacing $X$ with a dense Zariski open subset).  By \cref{jet lift} we may further replace $X$ with some jet space $J_AX$ and $\phi$ with $J_A\phi$ after passing to a Zariski open on which $J_AX\to X$ admits a section through every point.  
\begin{lemma}
Suppose $\G$ acts with finite kernel on a smooth irreducible algebraic variety $Y$.  Then for some jet space $J_AY$, the induced action of $\G$ on $J_AY$ has finite stabilizers.
\end{lemma}

\begin{proof}
Choose a point $y\in Y$ and let $\mathbf{S}_y$ be the stabilizer of $y$.  Then $\mathbf{S}_y$ acts with finite kernel on $\hat\cO_{Y,y}$, and therefore on some artinian quotient $\cO_{Y,y}/m_{y}^n$.  Take $A_n=\cO_{Y,y}/m_{y}^n$. Let $J^o_{A_n}Y$ denote the open subspace consisting of jets which are local isomorphisms. Note that all the points in a single fiber of  $J^o_{A_n}Y\ra Y$ have the same stabilizer. Thus if we let $Y_n$ denote those points of $Y$ whose fibers in $J^o_{A_n}Y$ have finite stabilizer, the $Y_n$ form an increasing chain of constructible sets whose union is all of $Y$. The claim follows by Noetherianity.

\end{proof}
Thus, passing to a $\G$-invariant Zariski open of an appropriate jet space we may assume the action of $\G$ on $Y$ has finite stabilizers.  The quotient $Y\to Z:=\G\backslash Y$ exists as an algebraic space, so by passing to a further $\G$-invariant Zariski open of $Y$ we may assume that there is a quotient $Y\to Z:=\G\backslash Y$ as a scheme.  The map $Y\to Z$ admits a section over a finite dominant map $Z'\to Z$, and therefore its base-change is a finite map $Y':=Y\times_Z Z'\to Y$ and $Y'$ is identified in a $\G$-equivariant way with $\G\times Z'$ where $\G$ acts by left multiplication on $\G$ and the identity on $Z'$.  After shrinking $X$, we may assume $Z'\to Z$ is in addition \'etale, and as $\tilde X$ is simply connected, the composition $\tilde X\xrightarrow{\phi} Y^\an\to Z^\an$ can be lifted through $Z'\to Z$, and the product map is a $\pi_1(X^\an,x)$-equivariant map $\phi':\tilde X\to Y'^\an$ which is definable by definable choice.  By \cref{et lift} we have $\bar{K_\partial(\phi)}=\bar{K_\partial(\phi')}$.  By definable Chow, functions pulled back from $Z'$ are in $K(X)\subset K_\Def(\tilde X)$, so we may assume $Y=\G$ with the left action by $\G$.  Finally, we have $\G\subset\mathbf{End}(U)$, so clearly we may assume $\G=\GLGL(\End(U))$ and $Y=\mathbf{End}(U)$ with the action of $\G$ by left multiplication.  We will show that in this case $K(V)=K_\partial(\phi)$.

As left $\G(\C)$-modules, with this action we have $\End(U)\cong U_0^
\vee\otimes_\C U$.  We then identify the coordinate ring $R[\mathbf{End}(U)]$ with $\Sym^*\End(U)^\vee$ as left $\G(\C)$-modules.  Thus, $R[\mathbf{End}(U)]$ admits a surjection (of $\G(\C)$-modules) from a direct some of tensor powers of $U^\vee$. It follows that the pullback of any function to $K_{\Def}(\tilde{X})$ is contained in a $\pi_1(X,x)$ local system which is a quotient of some tensor power of $U^
\vee$. 

It follows by \cref{basic RH} and \cref{lem: monofactor} that we have $K_\partial(\phi)\subset K(V)$.

On the other hand, by taking the linear coordinates we see that the condition in \cref{condition} is met.  Indeed, from the inclusion $\End(U)^\vee\subset R[\G]$, the natural map of left $\G(\C)$-modules
\[U^\vee \to\C[\G]\otimes_\C U_0^\vee\]
evaluates to $\mathrm{id}_{U_0^\vee}$ at the identity.  Thus $K(V)\subset K_\partial(\phi)$.

\end{proof}

\section{Periods of algebraic varieties}

In this section we relate the previous sections work on local systems to integrals of algebraic forms in families of varieties.

\subsection{Background on periods}  We begin with an informal but more detailed discussion of algebraic de Rham cohomology classes and their periods to explain how the main ideas of the previous sections can be used to understand them.

\subsubsection{Smooth de Rham cohomology}
Let $X$ be a smooth complex algebraic variety.  The smooth Poincar\'e resolution is acyclic so we may naturally think of the cohomology $H^k(X^\an,\C)$ as classes of closed smooth $k$-forms.
\subsubsection{Analytic de Rham cohomology} In the analytic category we have a resolution 
\[0\to \C_{X^\an}\to\cO_{X^\an}\xrightarrow{d}\Omega_{X^\an}^1\xrightarrow{d}\cdots\]
and therefore we obtain a natural isomorphism $\mathbb{H}^k(X^\an,\Omega^\bullet_{X^\an})\cong H^k(X^\an,\C)$ where
\[\Omega^\bullet_{X^\an}:=\left[\cO_{ X}\xrightarrow{d}\Omega^1_{X^\an}\xrightarrow{d}\cdots\xrightarrow{d} \Omega_{ X^\an}^{\dim X-1}\xrightarrow{d}\Omega^{\dim X}_{X^\an}\right].\]
We can think of the hypercohomology $\mathbb{H}^k(X^\an,\Omega^\bullet_{X^\an})$ in terms of \v{C}ech cohomology with respect to a covering by Stein open sets, and any \v Cech cocycle can be solved by a smooth form.  Alternatively, the Dolbeault resolution provides an acyclic resolution of $\Omega_{X^\an}^\bullet$, and in this way a hypercohomology class may be directly despresented by the class of a smooth form.  Either way, integrating such a form provides the explicit isomorphism.  

\subsubsection{Algebraic de Rham cohomology}
In the algebraic category we have a complex 
\[\Omega^\bullet_{X}:=\left[\cO_{ X}\xrightarrow{d}\Omega^1_{X}\xrightarrow{d}\cdots\xrightarrow{d} \Omega_{ X}^{\dim X-1}\xrightarrow{d}\Omega^{\dim X}_{X}\right].\]
By a theorem of Grothendieck, the natural map $\mathbb{H}^k(X,\Omega_X^\bullet)\to \mathbb{H}^k(X^\an,\Omega^\bullet_{X^\an})$ obtained by analytifying a \v Cech cocycle with respect to an affine open cover is an isomorphism.  We denote the composed isomorphism
\begin{equation}\label{comparison}\int_X:\mathbb{H}^k(X,\Omega_{X}^\bullet)\to H^k(X^\an,\C).\end{equation}
There are some circumstances in which algebraic de Rham cohomology classes can be thought of directly in terms of certain algebraic forms and for which the map $\int_X$ is the usual notion of integration along a cycle; see section \ref{sect integration} for some discussion.  In general the map can always be expressed in terms of integrating a \v{C}ech cocycle representative along the facets of a suitable singular cycle representative (see for instance Example \ref{eg 1forms}).  There is a natural map \[\ker\left(H^0(X,\Omega_X^k)\xrightarrow{d}H^0(X,\Omega_X^{k+1})\right)\to \mathbb{H}^k(X,\Omega_X^\bullet)\] 
associating to a closed algebraic $k$-form the obvious de Rham cohomology class, and in the image $\int_X$ is the usual integration map.
\subsubsection{De Rham cohomology in families}
Let $f:X\to S$ be an algebraic family which is real analytically locally trivial.  We can form the local system $V=R^k(f^\an)_*\C_{X^\an}$ whose fibers are the cohomologies $H^k(X_t^\an,\C)$, identified locally on $S^\an$ via a local trivialization of $f$.  There is an associated analytic flat vector bundle $(V_{\cO_{S^\an}},\nabla)$ which by the relative version of the analytic de Rham complex is computed as $R^k(f^\an)_*\Omega_{X^\an/S^\an}^\bullet$, where the connection is given by lifting vector fields and taking the Lie derivative.  The Lie derivative along an algebraic tangent field is algebraic, so $R^kf_*\Omega_{X/S}$ is naturally an algebraic flat vector bundle, and it is in fact the canonical algebraic structure on $V_{\cO_{S^\an}}$ guaranteed by the Riemann--Hilbert correspondence \cite{deligneRH}.  The  resulting isomorphism
\[\cO_{S^\an}\otimes_{\cO_{S^\an}} V_{\cO_S}\to \cO_{S^\an}\otimes_{\C_{S^\an}}V\]
fiberwise restricts to the integration isomorphism of section \eqref{comparison}.

We now describe the flat coordinates of algebraic sections.  Around any point $s\in S$ we can take a neighborhood $s\in U\subset S^\an$ and a trivialization $X^\an|_U\cong X_s\times U$ restricting to the identity at $s$ which we think of as a family of diffeomorphisms $f_u:X^\an_u\to X^\an_s$, and the flat continuation of a class $\beta\in H^k(X^\an,\C)$ is given by $\beta(u):=f_u^*\beta\in H^k(X^\an_u,\C)$.  Likewise we can continue a basis of cycles $\gamma_i\in H_k(X^\an_s,\C)$ as $\gamma_i(u):=(f_u^{-1})_*\gamma_i$.  The above isomorphism then maps an algebraic de Rham cohomology class $\alpha$ to
\[1\otimes \alpha\mapsto \sum_i \left(\int_{\gamma_i(u)}\alpha|_{X_u}\right)\otimes \gamma_i^\vee(u).  \]
The integrals $\int_{\gamma_i(u)}\alpha|_{X_u}$ therefore generate $K(V)$ for $V=R^k(f^\an)_*\C_{X^\an}$. Again, these functions can ultimately be interpreted in terms of integrals of algebraic forms, and at the very least we have the following:
\begin{corollary}\label{cor:geomperiods}
    Let $f:X\ra S$ be a family of smooth varieties which is real analytically locally trivial, $k\geq 0$, and let $V$ be the local system $R^k(f^\an)_*\C_{X^\an}$ on $S^{\an}$. Then the periods of a regular fiberwise closed $k$-form on $X$ along some locally constant $k$-dimensional homology class lie in $K(V)$.
\end{corollary}
The Hodge filtration $F^\bullet$ is in fact an algebraic filtration on $R^kf_*\Omega^\bullet_{X/S}$.  In the case that $f$ is smooth projective, it is given by truncating the complex $\Omega_{X/S}^\bullet$.

We shall need the following lemma for the proof of Theorem \ref{thm:intro main}
\def\g{\frak{g}}
\def\h{\frak{h}}
\def\H{\mathbf{H}}
\def\gr{\operatorname{gr}}
\begin{lemma}\label{lem:tensor}
Let $\G,\H$ be algebraic groups with $\G$ semisimple, and let $V$ be an indecomposable representation of $\G\times \H$. Then $V\cong V_1\otimes V_2$ for $V_1,V_2$ representations of $\G,\H$ respectively.
    
\end{lemma}

\begin{proof}
We first claim that $V$ is isoytopic as a representation of $\G$. Indeed, write $V=\oplus_{i=1}^m V^{(i)}$ where the $V^{(i)}$ are the isoytopic components of $V$. Then any $h\in \H$ gives an element of $\End_\G V$ and must therefore preserve the decomposition. Thus the $V^{(i)}$ are all $\G\times \H$ summands of $V$, and the claim follows since we assumed that $V$ is indecomposable.

Thus, there is an irreducible representation $V_1$ which is the only irreducible $\G$-subquotient of $V$, and since $\G$ is semisimple it follows that $V\cong V_1\otimes V_2$ as a $\G$-representation, with $\G$ acting trivially on $V_2$.  But now $\End_\G(V)\cong \End(V_2)$ and so the action of $\H$ is induced from an action of $\H$ on $V_2$. The claim follows.
    
\end{proof}

\begin{proof}[Proof of Theorem \ref{thm:intro main}]From the above discussion, algebraic de Rham cohomology is the canonical algebraic structure on $V$, and the periods $\int_{\gamma(u)}\alpha|_{X_u}$ for $\alpha$ an algebraic degree $k$ de Rham cohomology class and $\gamma(u)$ a flat degree $k$ cycle together generate $K(V)$.  The period domain coordinates are ratios of the periods, and $K(V)$ is differentially closed by \Cref{basic RH}, so we always have the claimed containment.  Theorem \ref{KV in Kphi} will then imply the result once we demonstrate that the conditions on the period map $\phi$ are met.  

The definability follows from \cite{bbkt}.  The algebraic group $\G$ can be taken to be the Zariski closure of the image of the monodromy representation of $V$.  Thus $\G$ is a normal subgroup of the derived group of the generic Mumford--Tate group \cite{andre}.  Moreover, the Zariski closure of the image of $\phi$ is a corresponding weak Mumford--Tate domain, namely a $\G$-orbit.

It remains to show that the action of $\G$ on the weak Mumford--Tate domain $\check D$ has finite kernel under the assumption in the theorem.  Let $\g$ be the Lie algebra of $\G$. 
 It is stable by the action of the generic Mumford--Tate group, hence underlies a sub-variation of Hodge structures of $\End(V)$. 
 The stabilizer of a point $p\in  \check D$ is $\exp(F^0\g)$, so the kernel $\H\subset\G$ of the action is the intersection of all conjugates of $\exp(F^0\g)$.  Its Lie algebra $\h$ is therefore an ideal $\h\subset \g$ which is contained in $F^0\g$.  The Lie algebra $\h$ is also a Hodge substructure (as it is again fixed by the generic Mumford--Tate group), and defined over $\bbR$, hence we must have $\h\cap W_{-1}\g=0$. 
 
 We claim that we have a  splitting $\g=\h\oplus \g'$ as Lie algebras and Hodge structures. Indeed, since $\gr_0^W\g$ is pure it must be semisimple, so we may write $\gr_0^W\g = \h \oplus \h'$ for some ideal $\h'$ which is also a hodge substructure. Now we may take $\g'$ to be the pre-image of $\h'$ in $\g$. 

 Moreover, since $\h$ is a pure Hodge structure of weight 0 contained in $F^0\g$ it must be Hodge--Tate, and is therefore unitary.  Therefore by \Cref{lem:tensor}, we contradict the assumption in the theorem unless $\h=0$ as desired.

\end{proof}

\subsection{Elliptic curves}
\subsubsection{Analytic presentation}
Take $S$ be the modular curve, so that $S^\an=\SL_2(\Z)\backslash \bbH$, and $f:X\to S$ to be the universal family of elliptic curves, so $X^\an=\Z^2\rtimes\SL_2(\Z)\backslash \C\times \bbH$, where the $\SL_2(\Z)$ action is 
$$\left(\begin{smallmatrix}
    a & b\\c & d
\end{smallmatrix}\right) \cdot (z,\tau) = \left(\frac{z}{c\tau+d},\frac{a\tau+b}{c\tau+d}\right).
$$

\subsubsection{Computing the 1-form}

On each fiber the form $dz$ is the unique regular 1-form up to scale, but this does not descend to $X^{\an}$. Instead, letting $\gamma = \left(\begin{smallmatrix}
    a & b\\c & d
\end{smallmatrix}\right)$ it satisfies $\gamma^*dz = \frac{dz}{c\tau+d}$. To remedy this, we use the $j$ function, which satisfies
$j'(\gamma\tau) = (c\tau+d)^2 j'(\tau)$, and so 
$j'(\gamma\tau)^{\frac12} = (c\tau+d) j'(\tau)^{\frac12}$, so the form  $\omega:=j'(\tau)^{\frac 12}dz$ does descend to $S^{\an}$.  (This is a multivalued function with finite monodromy so we must pass to a finite \'etale cover to obtain a regular 1-form, but we will take algebraic closures in the end so we will ignore this issue for now).
Moreover, this is a definable 1-form so by definable GAGA \cite{bbt} it gives an algebraic 1-form on $S$. The periods of $\omega$ are simply the integrals of $\omega$ from $0$ to $1$ and $0$ to $\tau$ so they are $j'(\tau)^{\frac12}$ and $\tau j'(\tau)^{\frac12}$. 

\subsubsection{Computing the two other periods}

Now there is another class in $H^1_{\dR}$ other than $[\omega]$. We could represent it by a meromorphic 1-form but this requires the elliptic $\wp$-function and integrating it is tricky. However, we can produce a new class by applying the connection $\nabla_{\partial/\partial j}[
\omega]$, and the resulting periods will be the derivatives of the periods of $[\omega]$. So we may simply differentiate the two periods we have so far to find two others which will span the 4-dimensional vector space of periods. Doing this gives

$$
R_1(\tau)=\frac{j''(\tau)}{j'(\tau)^{\frac 32}}
$$

$$
R_2(\tau) = \tau R_1(\tau) - \frac{2}{j'(\tau)^{\frac 12}}.
$$
One may in fact check that $\SL_2(\Z)$ acts on the vector space $\langle R_1,R_2\rangle$ giving the standard representation.
Finally, note that $j(\tau)$ is an algebraic co-ordinate on the modular curve.
Thus, the algebraic closure of the field of periods is $\ol{\C(\tau,j(\tau),j'(\tau),j''(\tau))}$. 

\subsubsection{Computing $K_\partial(\phi)$}

The period map here is simply the lift $\tilde{S^{\an}}\ra \bbH$ so the Grassmanian co-ordinate is $\tau$, and to compute $K_\partial(\phi)$ we have to differentiate with respect to $\tau$. Note that $\frac{d\tau}{dj} = j'(\tau)^{-1}$ and $\frac{dh}{dj}= \frac{dh}{d\tau}\frac{d\tau}{dj}$ it is from then on sufficient to differentiate with respect to $\tau$. Noting that $j'''$ is rational in $j,j',j''$ we conclude that $$K_\partial(\phi) = \C(\tau,j(\tau),j'(\tau),j''(\tau))$$ which has the same algebraic closure as the field of periods. 

\subsubsection{Necessity of the algebraic closure}

Working a bit more carefully, we may pass to an \'etale cover $Y$ of $S$ where $j'(\tau)^{\frac12}$ is holomorphic, and thus algebraic. Let $\C(Y)$ denote the field of rational functions on $Y$ which is a finite extension of $\C(j)$. Then by the above analysis the period field is 
$$\C(Y)(\tau,j'(\tau)^{\frac12},j''(\tau))$$ whereas
$$K_\partial(\phi) = \C(Y)(\tau,j'(\tau),j''(\tau)),$$ so we only have a containment of fields in general without taking closure.

\subsection{Interpreting the Riemann-Hilbert field through periods}\label{sect integration}In general it is somewhat complicated to describe the comparison between algebraic de Rham cohomology and singular cohomology in terms of integration.  In this section we describe some circumstances where there is a direct link.

The simplest example is provided by an affine family.

\begin{example}\label{affine eg}
Let $f:X\to S$ be a smooth family over a smooth affine base $S$ such that $V=R^kf_*\C_{X^\an}$ is a local system.  By general theory $V$ supports a variation of mixed Hodge structures.  If $X$ is affine then we have $R^kf_*\Omega^\bullet_{X/S}=f_*\mathscr{H}^k(\Omega_{X/S}^\bullet)=f_*Z^k_{X/S}/f_*B^k_{X/S}$ where
\begin{align*}
    Z^k_{Y/S}&=\ker\left(\Omega_{X/S}^k\xrightarrow{d}\Omega^{k+1}_{X/S}\right)\\
    B^k_{Y/S}&=\img\left(\Omega_{X/S}^{k-1}\xrightarrow{d}\Omega^{k}_{X/S}\right)
\end{align*}
since coherent sheaves have no higher cohomology.  Likewise for the analytic de Rham cohomology.  Thus, in this case the comparison
\[(R^kf_*\Omega_{X/S}^\bullet)^\an\to R^k(f^\an)_*\Omega_{X^\an/S^\an}^\bullet\to \cO_{S^\an}\otimes_{\C_{X^\an}}R^k(f^\an)_*\C_{X^\an}\]
is directly seen to be fiberwise integration, and $K(V)$ is generated by integrals of fiberwise closed $k$-forms on $X$ along $k$-cycles.  
\end{example}
\begin{example}
In fact, in the previous example, if we just assume $S$ is affine and $f$ is quasiprojective, by Jouanolou's trick \cite{Jouanolou} there is a vector bundle $E$ over $X$ and an $E$-torsor $Y$ whose total space is affine.  In particular the map $p:Y\to X$ is a homotopy equivalence upon analytification.  Thus, we have $R^kf_*\Omega_{X/S}^\bullet \cong R^kg_*\Omega^\bullet_{Y/S}$ where $g=f\circ p$, and so at the cost of changing the family the field $K(V)$ is still generated by integrals of closed algebraic $k$-forms along $k$-cycles.
\end{example}

\begin{example}\label{eg 1forms}In this example we explicitly describe the integration map on \v{C}ech cocycles computing algebraic de Rham cohomology in degree 1.  Let $X$ be a smooth algebraic variety.  The group $\mathbb{H}^1(X,\Omega_{X}^\bullet)$ can be described as follows:

Let $X_i$ be an affine cover of $X$.  Then $\mathbb{H}^1(X,\Omega^\bullet_X)$ is the quotient of the vector space generated by pairs $(f_{ij},\alpha_i)$ where $f_{ij}$ is a \v{C}ech 2-cocycle for $\cO_X$ and $\alpha_i$ is a \v Cech 1-cochain for $\Omega_X^1$ such that $df_{ij}=\alpha_i-\alpha_j$, modulo the vector subspace generated by pairs of the form $(f_i-f_j,0)$ for $f_i$ a \v Cech 1-cochain for $\cO_X$.  We claim that the integration of $(f_{ij},\alpha_i)$ along a 1-cycle $\gamma$ is obtained as follows:

We write $\gamma$ as a union of $1$-simplices (paths) $\bigcup_{i\in I} \gamma_i$ where each $\gamma_i$ is contained in a single $X_{b(i)}$. Pick an isomorphism $R:I\ra I$ such that $\gamma_i(0)=\gamma_{R(i)}(1)$. We then define the integral to be 
$$\sum_{i\in I}\left( \int_{\gamma_i}\alpha_{b(i)} + f_{iR(i)} (\gamma_i(0))\right).$$ 

It is easy to check that this is a well-defined map on cohomology, and that it is functorial. Since it is the correct map on $1$-forms (classes with all $f_{ij}=0)$ it follows that it is the correct map.

This can be generalized to higher dimensions, where one has to simplicially subdivide the integrating cycle and take the sum over all simplices of the integrals of the differential forms with certain rational coefficients.
\end{example}

\begin{example}
Let $C$ be a smooth proper curve of genus $g$, and let $U\subset C$ be an affine open subset, the complement of finitely many points.  Example \ref{affine eg} shows that every class in $H^1(U^\an,\C)$ is the smooth de Rham cohomology class of an algebraic $1$-form on $U$, and that the comparison $\int_C:\mathbb{H}^1(U,\Omega_{U/S}^\bullet)\to H^1(U^\an,\C)$ is given by integration in the usual sense. 

It is a classical fact that every class is represented by an algebraic $1$-form on $U$ with at worst logarithmic poles at $C\setminus C$.  We can also describe the Hodge filtration:  $F^pH^1(U^\an,\C)$ is $0$ for $p>1$, all of $H^1(U^\an,\C)$ if $p<1$, and $F^1H^1(U^\an,\C)$ is the subspace of de Rham cohomology classes of algebraic $1$-forms that extend regularly to $C$. 

The descriptions also hold in families. 

\end{example}

The final example can be generalized to higher dimensions.  Let $X$ be a smooth projective variety, $D$ an ample divisor, and $c=[D]\in H^2(X^\an,\C)$.  Recall that we have a Lefschetz decomposition:
\[H^k(X^\an,\Q)=\bigoplus_{0\leq j\leq \lfloor k/2\rfloor}c^{j}\cup H^{k-2j}(X^\an,\Q)_{\mathrm{prim}}\]
where $H^{n-k}(X^\an,\Q)_{\mathrm{prim}}:=\ker(c^{k+1}\cup\;\cdot\;)$.  

For a smooth projective family $f:X\to S$ with a relatively ample class $c$, we likewise have a direct sum decomposition of $R^k(f^\an)_*\Q_{X^\an}$, and the following proposition says that the Riemann--Hilbert field of the primitive part is generated by integrals of rational forms.

\begin{proposition}\label{ratl periods}Let $f:X\to S$ be a smooth projective family and let $D$ be a smooth relatively ample divisor of $X$ which is smooth over $S$.  Let $U=X\setminus D$ and $c=[D]\in H^0(S^\an,R^2(f^\an)_*\C_{X^\an})$.  Then for each $k\leq \dim X-\dim S$, the Riemann--Hilbert field of $V=(R^k(f^\an)_*\C_{X^\an})_{\mathrm{prim}}$ is generated by fiberwise integrals of algebraic $k$-forms on $U$ with at worst logarithmic poles along $D$ whose residue is exact. 
\end{proposition}

Proposition \ref{ratl periods} is fairly standard (see for example \cite{hodge1}), but we include the proof for completeness.

\begin{proof}Let $i:D\to X$ and $j:U\to X$ the inclusions, $g=f\circ j$, and $h=f\circ i$.  We have a short exact sequence of complexes 
$$0\ra \Omega_{X/S}^\bullet\ra\Omega_{X/S}^\bullet(\log D)\xrightarrow{\mathrm{Res}} i_*\Omega_{D/S}^\bullet[-1]\ra 0$$
whose analytification is canonically quasi-isomorphic to the exact triangle $f^{-1}\cO_{S^\an}\to Rj_*g^{-1}\cO_{S^\an}\to i_*i^!h^{-1}\cO_{S^\an}[1]\to f^{-1}\cO_{S^\an}[1]$.  The associated long exact sequence contains 
$$ R^{k-2}h_*\Omega^\bullet_{D/S}\ra R^kf_*\Omega_{X/S}^\bullet\ra R^kg_*\Omega^\bullet_{X/S}(\log D)\xrightarrow{\mathrm{Res}} R^{k-1}h_*\Omega^\bullet_{D/S}.$$ and the connecting homomorphism is identified with the pushforward map $i_*:H^{k-2}(D^\an,\C)\ra H^k(X^\an,\C)$.  By the Lefschetz hyperplane theorem, $R^kf_*\C_{X^\an}\to R^kf_*i_*\C_{D^\an}$ is an isomorphism for $k<\dim X-\dim S$.  Since the composition
\[R^{k-2}f_*\C_{X^\an}\xrightarrow{i^*}R^{k-2}f_*i_*\C_{D^\an}\xrightarrow{i_*} R^kf_*\C_{X^\an}\]
is cupping with $c$ by the projection formula, it follows that the image of $i_*$ is the imprimitive part of $R^kf_*\C_{X^\an}$, and therefore that the image of $R^kf_*\Omega^\bullet_{X/S}$ in $R^kg_*\Omega^\bullet_{U/S}$ analytifies to $\cO_{S^\an}\otimes_{\C_{S^\an}}V$.  On the other hand, $R^kg_*\Omega^\bullet_{U/S}=g_*Z_{X/S}^k/g_*B_{X/S}^k$ since $g$ is affine, and $\cO_{S^\an}\otimes_{\C_{S^\an}}V$ is therefore identified with the analytification of the subsheaf of $g_*Z_{U/S}^k/g_*B_{U/S}^k$ given by classes of regular forms on $U$ whose residues along $D$ are $0$ in $R^{k-1}h_*\Omega^\bullet_{D/S}$.
\end{proof}

\bibliography{biblio.quasi}
\bibliographystyle{plain}

\end{document}